\documentclass[12pt]{amsart}
\calclayout
\usepackage{amscd,amsmath}
\usepackage{graphicx}
\usepackage{amsfonts}
\usepackage{amsmath}
\usepackage{amssymb}
\newtheorem{teo}{Theorem}[section]
\newcommand\grad{\operatorname{grad}}
\renewcommand\div{\operatorname{div}}
\newcommand\curl{\operatorname{curl}}
\newcommand\tr{\operatorname{tr}}
\newcommand\diam{\operatorname{diam}}
\newcommand\F{\mathcal{F}}
\renewcommand\P{\mathcal{P}}
\newcommand\Q{\mathcal{Q}}
\newcommand\R{\mathbb{R}}
\renewcommand\S{\mathcal S}
\newcommand\T{\mathcal T}
\newcommand\x{\times}
\newcommand\Alt{\operatorname{Alt}}

\renewcommand{\>}{\rangle}
\newcommand{\pd}[2]{\frac{\partial#1}{\partial#2}}

\begin{document}

\title[Finite element differential forms on cubic meshes]{Finite element differential forms on curvilinear cubic meshes and
their approximation properties}
\author{Douglas N. Arnold}
\address{School of Mathematics,
University of Minnesota,
Minneapolis, MN 55455, USA}
\email{arnold@umn.edu}
\thanks{The first author was supported by NSF grant
DMS-1115291.}
\author{Daniele Boffi}
\address{Dipartimento di Matematica ``F. Casorati'',
Universit\`a di Pavia, 27100 Pavia, Italy}
\email{daniele.boffi@unipv.it}
\thanks{The second author was supported by IMATI-CNR, Italy and by MIUR/PRIN2009, Italy.}
\author{Francesca Bonizzoni}
\address{Faculty of Mathematics, University of Vienna, 1090 Wien, Austria.}
\email{francesca1.bonizzoni@univie.ac.at}
\keywords{finite element, differential form, tensor product, cubical meshes,
mapped cubical meshes, approximation}
\subjclass[2000]{65N30, 41A25, 41A10, 41A15, 41A63}

\begin{abstract}
We study the approximation properties of a wide class of finite element differential forms
on curvilinear cubic meshes in $n$ dimensions.  Specifically, we consider meshes in which each element
is the image of a cubical reference element under a diffeomorphism, and finite element
spaces in which the shape functions and degrees of freedom are obtained from the reference
element by pullback of differential forms.  In the case where the diffeomorphisms from
the reference element are all affine, i.e., mesh consists of parallelotopes, it is standard
that the rate of convergence in $L^2$ exceeds by one the degree of the largest full polynomial
space contained in the reference space of shape functions.  When the diffeomorphism is
multilinear, the rate of convergence for the same space of reference shape function
may degrade severely, the more so when the form degree is larger.  The main result of
the paper gives a sufficient condition on the reference shape functions to obtain a given
rate of convergence.  
\end{abstract}
\maketitle

\section{Introduction}\label{sc:intro}
Finite element simulations are often performed using meshes of
quadrilaterals in two dimensions or of hexahedra in three.  In simple
geometries the meshes may consist entirely of squares, rectangles,
or parallelograms, or their three-dimensional analogues, leading
to simple data structures.  However, for more general geometries,
a larger class of quadrilateral or hexahedral elements is needed,
a common choice being elements which are the images of a square or cube under an
invertible multilinear map.  In such cases, the shape functions
for the finite element space are defined on the square or cubic
reference element and then transformed to the deformed physical
element.  In \cite{DBF} it was shown in the case of scalar ($H^1$)
finite elements in two dimensions that, depending on the choice of
reference shape function space, this procedure may result in a loss
of accuracy in comparison with the accuracy achieved on meshes of
squares.  Thus, for example, the serendipity finite element space
achieves only about one half the rate of approximation when applied
on general quadrilateral meshes, as compared to what it achieves
on meshes of squares.  The results of \cite{DBF} were extended to
three dimensions in \cite{GM}.  The case of vector ($H(\div)$)
finite elements in two dimensions was studied in \cite{DBF2}.
In that case the transformation of the shape functions must be
done through the Piola transform and it turns out the same issue
arises, but even more strongly, the requirement on the reference
shape functions needed to ensure optimal order approximation being
more stringent.  Some results were obtained for $H(\curl)$
and $H(\div)$ finite elements in 3-D in \cite{FalkGattoMonk}.

The setting of finite element exterior calculus (see~\cite{D,DRR,DRR2})
provides a unified framework for the study of this problem.
In this paper we discuss the construction of finite element
subspaces of the domain $H\Lambda^k$ of the exterior derivative
acting on differential $k$-forms, $0\le k\le n$, in any number $n$
of dimensions.  This includes the case of scalar $H^1$ finite
elements ($k=0$), $L^2$ finite elements (in which the Jacobian
determinant enters the transformation, $k=n$), and, in three-dimensions,
finite elements in $H(\curl)$ ($k=1$) and $H(\div)$ ($k=2$).

The paper begins with a brief review of the relevant concepts
from differential forms on domains in Euclidean space.  We also
review, in Section~\ref{sc:tp}, the tensor product construction for differential
forms, and complexes of differential forms.  Although our main
result will not be restricted to elements for which the reference
shape functions are of tensor product type, such tensor product
elements will play an essential role in our analysis.
Next, in
Section~\ref{sc:fe}, we discuss the construction of finite element
spaces of differential forms and the use of reference domains and
mappings.  Combining these two constructions, in Section~\ref{sc:q},
we define the $\Q_r^-\Lambda^k$ finite element spaces, which may be
seen to be the most natural finite element subspace of $H\Lambda^k$
using mapped cubical meshes.  In Section~\ref{sc:approx}
we obtain the main result of the paper,
conditions for $O(h^{r+1})$ approximation in $L^2$
for spaces of differential $k$-forms on mapped cubic meshes.
More precisely, in the case of multilinear mappings Theorem~\ref{th:poly}
shows that a sufficient condition for $O(h^{r+1})$ approximation in $L^2$ is
that the reference finite element space contains $\Q_{r+k}^-\Lambda^k$.
Again we see loss of accuracy when the mappings are multilinear,
in comparison to affine (for which $\P_r\Lambda^k$ guarantees approximation of
order $r+1$), with the effect becoming more severe for
larger $k$.  In the final section we determine the extent of the
loss of accuracy for the $\Q_r^-\Lambda^k$ spaces. We discuss several examples
of finite element spaces, including standard $\Q_r$ spaces ($k=0$),
$\P_r\Lambda^n$ spaces, and serendipity spaces $\S_r\Lambda^k$ which have been
recently described in~\cite{serendipity}.

\section{Preliminaries on differential forms}\label{sc:prelim}
We begin by briefly recalling some basic notations, definitions, and properties of differential
forms defined on a domain $S$ in $\R^n$.  A differential $k$-form on $S$ is a function $S\to\Alt^k\R^n$,
where $\Alt^k\R^n$ denotes the space of alternating $k$-linear forms mapping $(\R^n)^k=\R^n\x\cdots\x\R^n\to\R$.
By convention
$\Alt^0\R^n=\R$, so differential $0$-forms are simply real-valued functions.  The space $\Alt^k\R^n$ has
dimension $\binom{n}{k}$ for $0\le k\le n$ and vanishes for other values of $k$.
A general differential $k$-form on $S$ can be expressed uniquely as
\begin{equation}\label{df}
f = \sum_{\sigma\in\Sigma(k,n)} f_\sigma\,dx^\sigma = \sum_{1\le \sigma_1\le\cdots\le \sigma_k\le n} f_{\sigma_1\cdots \sigma_n}\, dx^{\sigma_1}\wedge\cdots
\wedge dx^{\sigma_k},
\end{equation}
for some coefficient functions $f_\sigma$ on $S$, where $\Sigma(k,n)$ is the set of increasing maps from $\{1,\ldots,k\}$
to $\{1,\ldots,n\}$.
The wedge product $f\wedge g$ of a differential $k$-form
and a differential $l$-form is a differential $(k+l)$-form, and the exterior derivative $df$ of a
differential $k$-form is a differential $(k+1)$-form.  A differential $n$-form $f$ may be integrated
over an $n$-dimensional domain $S$ to give a real number $\int_S f$.

If $\F(S)$ is some space of real-valued functions on $S$, then we denote by
$\F\Lambda^k(S)$ the space of differential $k$-forms with coefficients
in $\F(S)$.  This space is naturally isomorphic to $\F(S)\otimes \Alt^k\R^n$.  Examples
are the spaces $C^\infty\Lambda^k(S)$ of smooth $k$-forms, $L^2\Lambda^k(S)$ of $L^2$ $k$-forms,
$H^r\Lambda^k(S)$ of $k$-forms with coefficients in a Sobolev space,
and $\P_r\Lambda^k(S)$ of forms with polynomial coefficients of degree at most $r$.
The space $L^2\Lambda^k(S)$ is a Hilbert space with inner product
$$
\<f,g\>_{L^2\Lambda^k(S)} = \sum_{\sigma\in\Sigma(k,n)} \<f_\sigma,g_\sigma\>_{L^2(S)},
$$
and similarly for the space $H^r\Lambda^k(S)$, using the usual Sobolev inner-product
$$
\<u,v\>_{H^r(S)} = \sum_{|\alpha|\le r}\<D^\alpha u,D^\alpha v\>_{L^2(S)}, \quad u,v\in H^r(S),
$$
where the sum is over multi-indices $\alpha$ of order at most $r$.

We may view the exterior derivative
as an unbounded operator $L^2\Lambda^k(S)\to L^2\Lambda^{k+1}(S)$ with domain
$$
H\Lambda^k(S):= \{\,f\in L^2\Lambda^k(S)\,|\, df\in L^2\Lambda^{k+1}(S)\,\},
$$
which is a Hilbert space when equipped with the graph norm
$$
\<f,g\>_{H\Lambda^k(S)} = \<f,g\>_{L^2\Lambda^k(S)} + \<df,dg\>_{L^2\Lambda^{k+1}(S)}.
$$
The complex
\begin{equation}\label{l2dR}
0\to H\Lambda^0(S)\xrightarrow{d} H\Lambda^1(S)\xrightarrow{d} \cdots
\xrightarrow{d} H\Lambda^m(S)\to 0
\end{equation}
is the $L^2$ de~Rham complex on $S$.

Let $\hat S$ be a domain in $\R^n$ and $F=(F^1,\ldots,F^n)$ a $C^1$ mapping of $\hat S$ into
a domain $S$ in some $\R^m$.
Given a differential $k$-form $v$ on $S$, its \emph{pullback} $F^*v$ is a
differential $k$-form on $\hat S$.
If
$$
v = \sum_{1\le i_1<\cdots <i_k\le m}v_{i_1\cdots i_k}\, dx^{i_1}\wedge\cdots\wedge dx^{i_k},
$$
then
\begin{equation}\label{pullback}
F^* v = \sum_{1\le i_1<\cdots <i_k\le m}\sum_{j_1,\ldots,j_k=1}^n 
  (v_{i_1\cdots i_k}\circ F) \pd{F^{i_1}}{\hat x^{j_1}}\cdots\pd{F^{i_k}}{\hat x^{j_k}}\,
  d\hat x^{j_1}\wedge\cdots\wedge d\hat x^{j_k}.
\end{equation}
The pullback operation satisfies $(G\circ F)^*= F^*\circ G^*$, and, if $F$ is a diffeomorphism,
$(F^*)^{-1} = (F^{-1})^*$.  The pullback preserves the wedge product and the exterior derivative
in the sense that
$$
F^*(v\wedge w) = F^*v\wedge F^*w, \quad dF^*v = F^*dv.
$$
If $F$ is a diffeomorphism of $\hat S$ onto $S$ which is orientation preserving (i.e., its Jacobian
determinant is positive), then the pullback preserves the integral as well:
$$
\int_{\hat S} F^*v = \int_S v, \quad v\in L^1\Lambda^n(S).
$$

An important application of the pullback is to define the trace of a differential form on
a lower dimensional subset.  If $S$ is a domain in $\R^n$ and $f$ is a subset of the closure
$\bar S$ and also an open subset of a hyperplane
of dimension $\le n$ in $\R^n$, then the pullback of the inclusion map $f\hookrightarrow S$
defines the trace operator $\tr_f$ taking $k$-forms on $\bar S$ to $k$-forms on $f$.

Before continuing, we recall that vector proxies exist for differential $1$-forms and $(n-1)$-forms
on a domain in $\R^n$.  That is, we may identify the $1$-form $\sum_i v_i\,dx^i$ with the
vector field $(v_1,\ldots,v_n):S\to\R^n$, and similarly we may identify the $(n-1)$-form
$$
\sum_i (-1)^i v_i \,dx^1\wedge\cdots \widehat {dx^i}\cdots\wedge dx^n
$$
with the same vector field, where the hat is used to indicate a suppressed argument.
A $0$-form is a scalar function, and an $n$-form can be identified with its coefficient which is
a scalar function.  Thus, on a three-dimensional domain all the spaces entering the de~Rham complex \eqref{l2dR}
have vector or scalar proxies and we may write the complex in terms of these:
$$
0\to H^1(S)\xrightarrow{\grad} H(\curl,S) \xrightarrow{\curl}
H(\div,S)\xrightarrow{\div} L^2(S) \to 0.
$$
Returning to the $n$-dimensional case,
the pullback of a scalar function $f$, viewed as a $0$-form, is just the composition:
$\hat v(\hat x)=v\bigl(F(\hat x)\bigr)$, $\hat x\in \hat S$.  If we identify
scalar functions with $n$-forms, then the pullback is $\hat v(\hat x) = \det[DF(\hat x)]v\bigl(F(\hat x)\bigr)$
where $DF(\hat x)$ is the Jacobian matrix of $F$ at $\hat x$.
For a vector field $v$, viewed as a $1$-form or an $(n-1)$-form, the pullback operation corresponds to
$$
\hat v(\hat x) = [DF(\hat x)]^T v\bigl(F(\hat x)\bigr), \quad
\hat v(\hat x)= \operatorname{adj}[DF(\hat x)] v\bigl(F(\hat x)\bigr),
$$
respectively.  The adjugate matrix, $\operatorname{adj} A$, is the transposed cofactor matrix, which is equal to $(\det A)A^{-1}$ in
case $A$ is invertible.  The latter formula, representing the pullback of an $(n-1)$-form, is called
the Piola transform.

Next we consider how Sobolev norms transform under pullback.
If $F$ is a diffeomorphism of $\hat S$ onto $S$ smooth up to the boundary,
then each Sobolev norm of the pullback $F^* v$ can be bounded in terms of the corresponding Sobolev
norm of $v$ and bounds on the partial derivatives of $F$ and $F^{-1}$.  Specifically, we have the
following theorem.
\begin{teo}\label{th:cov} Let $r$ be a non-negative integer and $M>0$.  There exists
 a constant $C$ depending only on $r$, $M$, and the dimension $n$, such that
$$
\|F^* v\|_{H^r\Lambda^k(\hat S)} \le C\|v\|_{H^r\Lambda^k(S)}, \quad v\in H^r\Lambda^k(S),
$$
whenever $\hat S,S$ are domains in $\R^n$ and $F:\hat S\to S$ is a $C^{r+1}$ diffeomorphism
satisfying
$$
\max_{1\le s\le r+1} |F|_{W^s_\infty(\hat S)}\le M,
\quad | F^{-1}|_{W^1_\infty(S)}\le M.
$$
\end{teo}
\begin{proof}
 From \eqref{pullback},
$$
\|F^* v\|_{H^r\Lambda^k(\hat S)}
=
\sum_{1\le i_1<\cdots <i_k\le m}\sum_{|\alpha|\le r}
\int_{\hat S} |D^\alpha[(v_{i_1\cdots i_k}\circ F) \pd{F^{i_1}}{\hat x^{j_1}}\cdots\pd{F^{i_k}}{\hat x^{j_k}}](\hat x)|^2\,
 d\hat x.
$$
Using the Leibniz rule and the chain rule, we can bound the integrand by
$$
C\sum_{|\beta|\le r}|(D^\beta v_{i_1\cdots i_k})\bigl( F(\hat x)\bigr)|
$$
where $C$ depends only  on $\max_{1\le s\le r+1} | F|_{W^s_\infty(\hat S)}$
and so can be bounded in terms of $M$.
Changing variables to $x=F\hat x$ in the integral brings in a factor of the Jacobian determinant
of $F^{-1}$, which is also
bounded in terms of $M$, and so gives the result.
\end{proof}

A simple case is when $F$ is a dilation:
$F(\hat x)=h \hat x$ for some $h>0$.  Then \eqref{pullback} becomes
$$
(F^* v)(\hat x) = \sum_{\sigma\in\Sigma(k,n)} v_\sigma(h\hat x) h^k\,
d\hat x^\sigma,
$$
and therefore,
$$
D^\alpha (F^*v)(\hat x) = \sum_{\sigma\in\Sigma(k,n)}  (D^\alpha v_\sigma)(h\hat x) h^{r+k}\,
d\hat x^\sigma,
$$
where $r=|\alpha|$.
We thus get the following theorem.
\begin{teo}\label{th:dil} Let $F$ be the dilation $F\hat x = h\hat x$ for some 
$h>0$, $\hat S$ a domain in $\R^n$, $S=F(\hat S)$, and $\alpha$ a multi-index of order $r$.
Then
$$
\|D^\alpha(F^* v)\|_{L^2\Lambda^k(\hat S)} = h^{r+k-n/2}\|D^\alpha v\|_{L^2\Lambda^k(S)}, \quad v\in H^r\Lambda^k(S).
$$
\end{teo}

\section{Tensor products of complexes of differential forms}
\label{sc:tp}
In this section we discuss the tensor product operation on differential forms and
complexes of differential forms.
The tensor product of a differential $k$-form on some domain and a differential $l$-form
on a second domain may be naturally realized as a differential $(k+l)$-form on the
Cartesian product of the two domains.  When this construction is combined with the standard
construction of the tensor product of complexes \cite[Section~5.7]{Hilton-Wylie}, we are led to a realization of the
tensor product of subcomplexes of the de ~Rham subcomplex on two domains as a subcomplex of
the de~Rham complex on the Cartesian product of the domains (see
also~\cite{Christiansen2009,Christiansen2011}).

We begin by identifying the tensor product of algebraic forms on Euclidean spaces.
For $m,n\ge 1$, consider the Euclidean space $\R^{m+n}$ with coordinates denoted
$(x_1,\ldots,x_m,y_1,\ldots,y_n)$.  The projection $\pi_1:\R^{m+n}\to\R^m$
on the first $m$ coordinates defines, by pullback, an injection
$\pi_1^*:\Alt^k\R^m\to\Alt^k\R^{m+n}$, with the embedding of
$\pi_2^*:\Alt^l\R^n\to\Alt^l\R^{m+n}$ defined similarly.
Therefore we may define a bilinear map
$$
\Alt^k\R^m\x\Alt^l\R^n\to\Alt^{k+l}\R^{m+n}, \quad
(\mu,\nu)\mapsto \pi_1^*\mu \wedge \pi_2^*\nu,
$$
or, equivalently, a linear map
$$
\Alt^k\R^m\otimes\Alt^l\R^n\to\Alt^{k+l}\R^{m+n}, \quad
\mu\otimes\nu\mapsto \pi_1^*\mu \wedge \pi_2^*\nu.
$$
This map is an injection.  Indeed, a basis for $\Alt^k\R^m\otimes\Alt^l\R^n$
consists of the tensors $dx^\sigma\otimes dy^\tau$ with
$\sigma\in\Sigma(k,m)$ and $\tau\in\Sigma(l,n)$, which simply maps
to $dx^\sigma\wedge dy^\tau$, an element of the standard basis
of $\Alt^{k+l}\R^{m+n}$.  In view of this injection, we may view
the tensor product of $\mu\otimes\nu$ of a $k$-form $\mu$ on $\R^m$ and
an $l$-form on $\R^n$ as a $(k+l)$-form on $\R^{m+n}$ (namely,
we identify it with $\pi_1^*\mu \wedge \pi_2^*\nu$).

Next we turn to differential forms defined on domains in Euclidean space.
Let $u$ be a differential $k$-form on a domain $S\subset\R^m$ and $v$
a differential $l$-form on $T\subset\R^n$.  We may identify the tensor product
$u\otimes v$ with the differential
$(k+l)$-form $\pi_S^*u \wedge \pi_T^*v$ on $S\x T$ where
$\pi_S:S\x T\to S$ and $\pi_T:S\x T\to T$ are the canonical projections.
In coordinates, this identification is
$$
(\sum_\sigma f_\sigma\,dx^\sigma)
\otimes(\sum_\tau g_\tau\,dy^\tau)
=\sum_{\sigma,\tau}
f_\sigma\otimes g_\tau \, dx^\sigma\wedge dy^\tau.
$$
Note that the exterior derivative of the tensor product is
\begin{equation}\label{dtp}
 \begin{aligned}
 d_{S\x T}(\pi_S^* u \wedge \pi_T^* v) &= d_{S\x T}(\pi_S^* u)\wedge\pi_T^* v
+(-1)^k\pi_S^* u \wedge d_{S\x T}(\pi_T^* v)
\\
&= \pi_S^*(d_S u)\wedge\pi_T^* v
+(-1)^k\pi_S^* u \wedge \pi_T^*(d_T v).
\end{aligned}
\end{equation}

Having defined the tensor product of differential forms,
we next turn to the tensor product of complexes of differential forms.
A subcomplex of the $L^2$ de~Rham complex \eqref{l2dR},
\begin{equation}\label{complV}
0\to V^0 \xrightarrow{d} V^1\xrightarrow{d}\cdots\xrightarrow{d} V^m\to 0,
\end{equation}
is called a de~Rham subcomplex on $S$.  This means that, for each $k$,
$V^k\subset H\Lambda^k(S)$
and $d$ maps $V^k$ into $V^{k+1}$.
Suppose we are given such a de~Rham subcomplex on $S$ and also a de~Rham subcomplex,
\begin{equation}\label{complW}
0\to W^0 \xrightarrow{d} W^1\xrightarrow{d}\cdots\xrightarrow{d} W^n\to 0,
\end{equation}
on $T$.  The tensor product of the complexes \eqref{complV} and \eqref{complW} is
the complex
\begin{equation}
\label{complVW}
0\to (V\otimes W)^0 \xrightarrow{d} (V\otimes W)^1\xrightarrow{d}\cdots\xrightarrow{d} (V\otimes W)^{m+n}\to 0,
\end{equation}
where
\begin{equation}\label{tpspace}
(V\otimes W)^k:=\bigoplus_{i+j=k} (V^i\otimes W^j),\quad k=0,\ldots,m+n,
\end{equation}
and the differential $(V\otimes W)^k\to (V\otimes W)^{k+1}$ is defined by
$$
d(u\otimes v) = d_S^i u\otimes v +(-1)^i u\otimes d_T^j v,\quad
u\in V^i, v\in W^j.
$$
In view of the identification of the tensor product of differential forms
with differential forms on the Cartesian product, the space $(V\otimes W)^k$ in
\eqref{tpspace} consists of differential $k$-forms on $S\x T$, and, in
view of \eqref{dtp}, the differential is the
restriction of the exterior derivative on $H\Lambda^k(S\x T)$.  Thus
the tensor product complex \eqref{complVW} is a subcomplex of the
de~Rham complex on $S\x T$.

\section{Finite elements}\label{sc:fe}
Let $\Omega$ be a bounded domain in $\R^n$, $0\le k\le n$.  As in \cite{C},
a finite element space of $k$-forms on $\Omega$ is assembled from several ingredients: a
\emph{triangulation}
$\T$ of $\Omega$, whose elements we call \emph{finite elements}, and, for each finite element $K$, a space $V(K)$ of
\emph{shape functions} on $K$, and a set
$\Xi(K)$ of \emph{degrees of freedom}.  We now describe these ingredients more precisely.

For the triangulation, we allow the finite elements to be curvilinear polytopes.  This
means that each $K$ is the image $F_K(\hat K)$ of an $n$-dimensional polytope $\hat K$ (so a closed
polygon in two dimensions and a closed polyhedron in three dimensions) under a smooth invertible map $F_K$
of $\hat K$ into $\R^n$.  The faces of $K$ are defined as the images of the faces of $\hat K$, and the
requirement that $\T$ be a triangulation of $\Omega$ means that $\bar \Omega=\bigcup_{K\in\T} K$
and that the intersection of any two elements of $\T$ is either empty or is a common face of both,
of some dimension.

The space $V(K)$ of shape functions is a finite-dimensional space of $k$-forms on $K$.  The degrees
of freedom are a unisolvent set of functionals on $V(K)$, or, otherwise put, $\Xi(K)$ is
a basis for the dual space $V(K)^*$.  Further, each degree of freedom is associated to a specific face
of $K$, and when two distinct elements $K_1$ and $K_2$ intersect in a common face $f$, the
degrees of freedom of $K_1$ and $K_2$ associated to the face $f$ are in $1$-to-$1$ correspondence.

The finite element functions are then defined as the differential forms on $\Omega$ which belong
to the shape function spaces piecewise, and for which corresponding degrees of freedom are single-valued.
That is, to define a finite element function $u$ we specify, for all elements $K$, shape functions $u_K\in V(K)$
satisfying $\xi_1(u_{K_1})=\xi_2(u_{K_2})$
whenever $K_1$ and $K_2$ share a common face and $\xi_1\in\Xi_1(K)$ and $\xi_2\in\Xi_2(K)$
are corresponding degrees of freedom associated to the face.
Then $u$ is defined almost everywhere by setting its restriction
to the interior of each element $K$ to be $u_K$.  The finite element space $V(\T)$ is
defined to be the space of all such finite element functions.

The degrees of freedom associated to faces of dimension $<n$ determine the inter-element
continuity imposed on the finite element functions.  Specifying the continuity in
this way leads to finite element spaces which can be efficiently implemented.  We note that the
finite element space is unchanged if we use a different set of degrees of freedom, as long as
the span of the degrees of freedom associated to each face is unchanged.  Thus we shall usually
specify these spans, rather than a particular choice of basis for them.

When a finite element $K$ is presented as $F_K(\hat K)$ using a \emph{reference element} $\hat K\subset\R^n$
and a diffeomorphic mapping $F_K:\hat K\to K$ (as for curvilinear polytopes), it is often convenient
to specify the shape functions and degrees of freedom in terms of the reference element
and the mapping.  That is, we specify a finite dimensional space $V(\hat K)$ of differential
$k$-forms on $\hat K$, the \emph{reference shape functions}, and define $V(K)=(F_K^{-1})^*V(\hat K)$, the pullback under
the diffeomorphism $F_K^{-1}:K\to \hat K$.  Similarly, given a space $\Xi(\hat K)$ of
degrees of freedom for $V(\hat K)$ on $\hat K$ and a reference
degree of freedom $\hat\xi\in \Xi(\hat K)$ associated to a face
$\hat f$ of $\hat K$, we define $\xi$ by $\xi (v) = \hat\xi(F_K^* v)$ for $v\in V(K)$ and associate
$\xi$ to the face $f=F_K(\hat f)$ of $K$.  In this way we determine the degrees of freedom $\Xi(K)$.
In the present paper we shall be concerned with the case where each element of the mesh is
presented as the image of the unit cube $\hat K=[0,1]^n$  under a multilinear
map (i.e., each component of $F_K$ is a polynomial of degree at most one in each of the $n$ variables).
A special case is when $F_K$ is affine, in which case the elements $K$ are arbitrary parallelotopes
(the generalization of parallelograms in two dimensions and parallelepipeds in three dimensions).  In the general
case, the elements are arbitrary convex quadrilaterals
in two dimensions, but
in more dimensions they may be truly curvilinear: the edges are straight, but the
faces of dimension two need not be planar.

We need some measure of the shape regularity of an element $K$.  Since the shape
regularity should depend on the shape, but not the size, of the element, we require
it to be invariant under dilation.  Therefore, let $h_K=\diam(K)$ and
set $\bar K= h_K^{-1} K$, which is of unit diameter.  Assuming only that the
element is a Lipschitz domain, the Bramble--Hilbert lemma ensures that for $r\ge 0$ integer
there exists a constant $C$ such that
\begin{equation}\label{bh}
\inf_{p\in\P_r(\bar K)}\|u-p\|_{L^2(\bar K)}\le C |u|_{H^{r+1}(\bar K)}, \quad u\in H^{r+1}(\bar K).
\end{equation}
We define $C(K,r)$ to be the least such constant, and take this as our shape regularity
measure.  If the domain $K$ is convex then $C(K,r)$ can be bounded in terms only of $r$ and
the dimension $n$; see \cite{verfurth}.  If the
domain is star-shaped with respect to a ball of diameter at least $\delta h_K$ for some $\delta>0$,
then $C(K,r)$ can be bounded in terms of $r$, $n$, and $\delta$; see \cite{duran}.

The next theorem will play a fundamental role in proving approximation
properties of finite element differential forms.

\begin{teo}\label{th:jackson}
 Let $K$ be a bounded Lipschitz domain in $\R^n$ with diameter $h_K$ and let
 $r$
be a non-negative integer.  Then there exists a constant $C$ depending only
on $r$, $n$, and
the shape constant $C(K,r)$, such that
$$
\inf_{p\in\P_r\Lambda^k(K)}\|u-p\|_{L^2\Lambda^k(K)}\le C
h_K^{r+1}|u|_{H^{r+1}\Lambda^k(K)}, \quad u\in H^{r+1}\Lambda^k(K),
$$
for $0\le k\le n$.
\end{teo}
\begin{proof}
Let $\bar K=h_K^{-1} K$,  let $F:\bar K\to K$ be the dilation, and set $\bar
u=F^*u$.  Applying \eqref{bh} to each component of $F^*u$, we obtain $p\in \P_r\Lambda^k(\bar K)$ with
\begin{equation}\label{t}
\|F^*u-p\|_{L^2\Lambda^k(\bar K)}\le C |F^*u|_{H^{r+1}\Lambda^k(\bar K)}.
\end{equation}
Set $q = (F^{-1})^*p$.  Since $F^{-1}$ is a dilation, $q\in
\P_r\Lambda^k(K)$, and clearly
$F^*(u-q)=F^*u-p$.  Combining \eqref{t} and Theorem~\ref{th:dil} gives the
desired estimate.
\end{proof}

\section{Tensor product finite element differential forms on cubes}\label{sc:q}
In this section, we describe the construction of specific finite element spaces of differential forms.
We suppose that the triangulation consists of curvilinear cubes, so that each element $K$ of
the triangulation $\T$ is specified by a smooth diffeomorphism $F_K$ taking the unit
cube $\hat K=[0,1]^n$ onto $K\subset \R^n$.  As described above, the shape functions and degrees
of freedom on $K$ can then be determined by specifying a space $V(\hat K)$
of reference shape functions and a set $\Xi(\hat K)$ of reference degrees of freedom.
Now we apply the tensor product construction of Section~\ref{sc:tp} to construct the
reference shape functions and degrees of freedom.

Fix an integer $r\ge 1$.  With $I=[0,1]$ the unit interval,
let $\P_r\Lambda^0(I)=\P_r(I)$ denote the space of polynomial functions ($0$-forms)
on $I$ of degree at most $r$, and let $\P_{r-1}\Lambda^1(I)=\P_{r-1}(I)\,dx$ denote the space of polynomial $1$-forms
of degree at most $r-1$.
Connecting these two spaces with the exterior derivative, we obtain a subcomplex of the
de~Rham complex on the unit interval:
\begin{equation}\label{q1d}
\P_r(I) \xrightarrow{d} \P_{r-1}(I)\,dx.
\end{equation}
Taking the tensor product of this complex with itself, we obtain a subcomplex
of the de~Rham complex on $I^2$, whose spaces we denote by $\Q_r^-\Lambda^k=\Q_r^-\Lambda^k(I^2)$:
\begin{equation}\label{q2d}
\Q_r^-\Lambda^0 \xrightarrow{d} \Q_r^-\Lambda^1\xrightarrow{d} \Q_r^-\Lambda^2.
\end{equation}
Specifically, writing $\P_{r,s}=\P_{r,s}(I^2)$ for $\P_r(I)\otimes\P_s(I)$, we have on $I^2$,
$$
\Q_r^-\Lambda^0 = \P_{r,r}, \quad
\Q_r^-\Lambda^1 = \P_{r-1,r}\,dx \,\oplus\, \P_{r,r-1}\,dy, \quad
\Q_r^-\Lambda^2 = \P_{r-1,r-1}\,dx\wedge dy.
$$
The first space $\Q_r^-\Lambda^0$ is the tensor product polynomial space
traditionally denoted $\Q_r$, and the last space is $\Q_{r-1}\,dx\wedge dy$.

To get spaces on the 3-D cube,
we may further take the tensor product of the 2-D complex \eqref{q2d} with the 1-D complex \eqref{q1d},
or, equivalently, take the tensor product of three copies of the 1-D complex, to obtain
a de~Rham subcomplex on the unit cube $I^3$:
$$
\Q_r^-\Lambda^0 \xrightarrow{d}
\Q_r^-\Lambda^1 \xrightarrow{d}
\Q_r^-\Lambda^2 \xrightarrow{d}
\Q_r^-\Lambda^3.
$$
Here $\Q_r^-\Lambda^0 =\Q_r$, $\Q_r^-\Lambda^3 = \Q_{r-1}\,dx\wedge dy\wedge dz$, and
\begin{gather*}
\Q_r^-\Lambda^1 = \P_{r-1,r,r}\,dx \,\oplus\, \P_{r,r-1,r}\,dy \,\oplus\, \P_{r,r,r-1}\,dz, \\
\Q_r^-\Lambda^2 = \P_{r,r-1,r-1}\,dy\wedge dz \,\oplus\, \P_{r-1,r,r-1}\,dx\wedge dz \,\oplus\, \P_{r-1,r-1,r}\,dx\wedge dy.
\end{gather*}
The extension of this construction to higher dimensions is clear,
yielding spaces $\Q_r^-\Lambda^k(I^n)$ for $0\le k \le n$, $n\ge 1$,
spanned by quantities $p\, dx^\sigma$, where $\sigma\in\Sigma(k,n)$ and $p$ is a polynomial of
degree at most $r$ in all variables and of degree at most $r-1$ in the
variables $x^{\sigma_i}$. More precisely,
$$
\Q_r^-\Lambda^k(I^n) = \bigoplus_{\sigma\in
\Sigma(k,n)}\left[\bigotimes_{i=1}^n \P_{r-\delta_{i,\sigma}}(I)\right]\,
dx^{\sigma_1}\wedge\cdots\wedge dx^{\sigma_k},
$$
where
$$
\delta_{i,\sigma} =
\begin{cases}
 1, & i\in\{\sigma_1,\ldots,\sigma_k\},\\
0, & \text{otherwise}.
\end{cases}
$$
In the case $k=0$, this space is understood to be
$$
\Q_r^-\Lambda^0(I^n) = \bigotimes_{i=1}^n \P_r(I),
$$
i.e., the space conventionally referred to as $\Q_r(I^n)$.  This definition
makes sense also
if $r=0$, so $\Q_0^-\Lambda^0(I^n)=\Q_0(I^n)=\R$ is the space of constant
functions.  For
$k>0$, $\Q_0^-\Lambda^k(I^n)=0$.
We also interpret this in the case $n=0$, so $I^n$ is a single point.  Then
$\Q_r^-\Lambda^0(I^0)$
is understood to be the space $\R$ of constants. It is easy to check that
$$
\dim \Q_r^-\Lambda^k(I^n) = \binom{n}{k}(r+1)^{n-k}r^k,
$$
in all cases $0\le k\le n$, $r\ge 0$.

Let us now characterize the degrees of freedom of $\Q_r^-\Lambda^k(I^n)$.
The space $\P_r\Lambda^0(I)=\P_r(I)$ has dimension $r+1$.  It has one degree of freedom associated
to each of the two vertices $p=0$ and $p=1$, namely the evaluation functional $v\mapsto v(p)$. The
remaining $r-1$ degrees of freedom are associated to $I$ itself, and are given by
$$
v \mapsto \int_I v(x)q(x)\,dx, \quad q\in\P_{r-2}(I).
$$
The space $\P_{r-1}\Lambda^1(I)=\P_{r-1}(I)$ has dimension $r$ and its degrees of
freedom are given by 
$$
v \mapsto \int_I v(x)q(x)\,dx, \quad q\in\P_{r-1}(I).
$$
The tensor product construction then yields degrees of freedom for $\Q_r^-\Lambda^k(I^n)$.
Namely, each degree of freedom for $\Q_r^-\Lambda^k(I^n)$ is of the form
$\xi_1\otimes\ldots\otimes\xi_n$, and is associated to the Cartesian product
face $f_1\times\ldots\times f_n$, where $\xi_j\in \Xi(I)$ is associated to the
face $f_j$ of the interval $I$  for all $j$.
Hence, a set of unisolvent degrees of freedom for  $\Q_r^-\Lambda^k(I^n)$ ($r\ge1$, $0\le k\le n$)
is given by
\begin{equation}\label{dofs}
v \mapsto \int_f \tr_f v(x) \wedge q(x), \quad q\in \Q_{r-1}^-\Lambda^{d-k}(f),
\end{equation}
for each face $f$ of $I^n$ of degree $d\ge k$.

Let us count that we have supplied the
correct number of degrees of freedom.  Since the number of faces of dimension
$d$ of $I^n$
is $2^{n-d}\binom{n}{d}$, the total number of degrees of freedom is
$$
\sum_{d=k}^n 2^{n-d}\binom{n}{d}\binom{d}{k}r^k(r-1)^{d-k}.
$$
Substituting $\binom{n}{d}\binom{d}{k}=\binom{n}{k}\binom{n-k}{n-d}$
and then changing the summation index from $d$ to $m=n-d$ we get:
\begin{align*}
\sum_{d=k}^n 2^{n-d}\binom{n}{k}\binom{n-k}{n-d}r^k(r-1)^{d-k}
&=\binom{n}{k}r^k\sum_{m=0}^{n-k}\binom{n-k}{m} 2^m(r-1)^{n-k-m}\\
&=\binom{n}{k}r^k(r-1+2)^{n-k}=\dim\Q_r^-\Lambda^k(I^n),
\end{align*}
as desired.

Figure~\ref{fg:q}, taken from \cite{fe-families}, shows degree of freedom diagrams for the $\Q_r^-\Lambda^k$ spaces in two and three dimensions.
In such diagrams, the number of symbols drawn in the interior of a face is equal to the number of degrees of freedom
associated to the face.

\begin{figure}[p]
\begin{center}
\begin{tabular}{cccc}
\hspace{-1.35in}$\Q_r^-\Lambda^k$ (2D): & $k=0$ & $k=1$ & $k=2$
\\[.15in]
\raise.5in\hbox{$r=1$}
  & \includegraphics[width=1.2in]{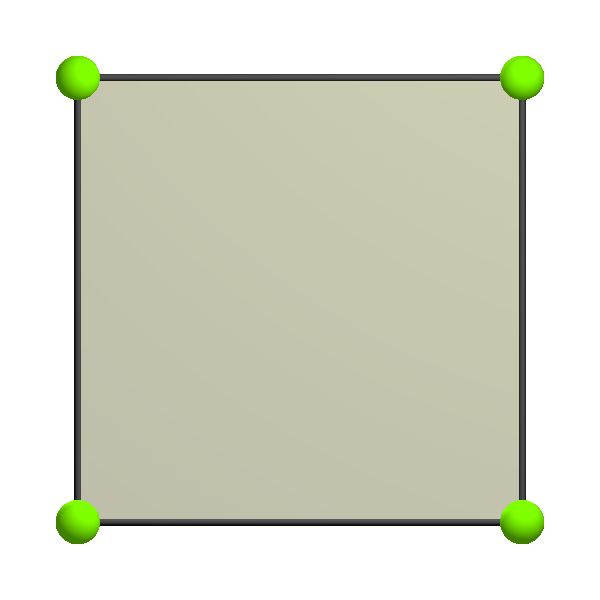}
  & \includegraphics[width=1.2in]{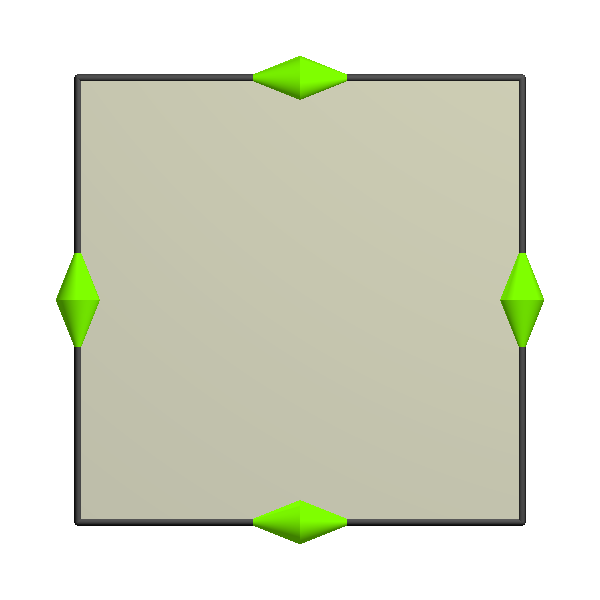}
  & \includegraphics[width=1.2in]{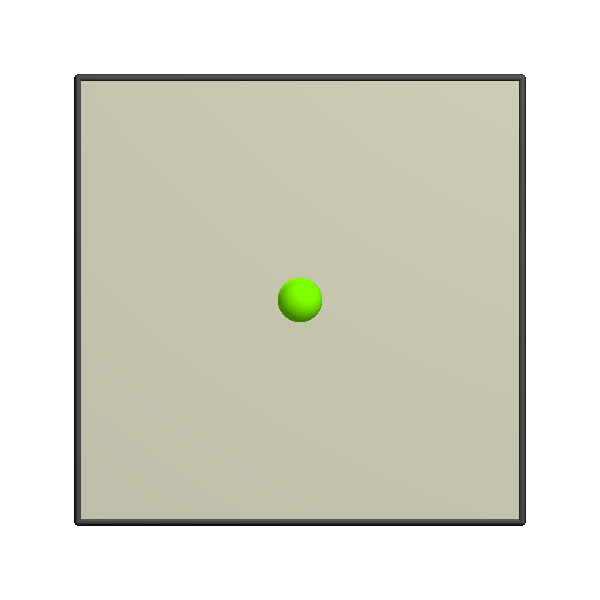}
\\
\raise.5in\hbox{$r=2$}
  & \includegraphics[width=1.2in]{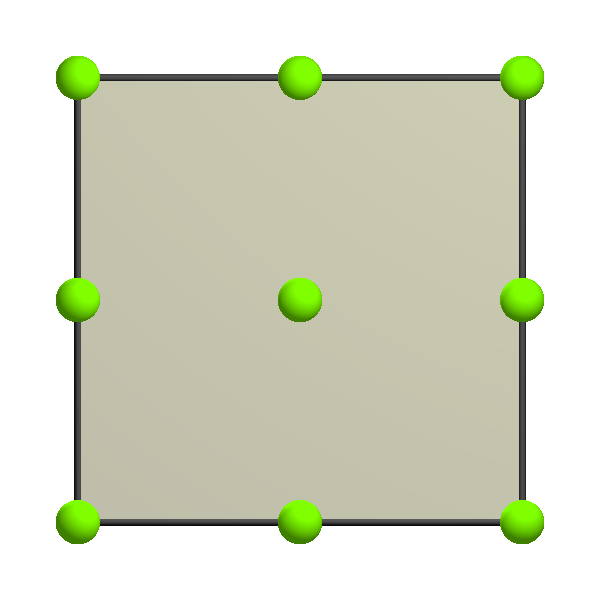}
  & \includegraphics[width=1.2in]{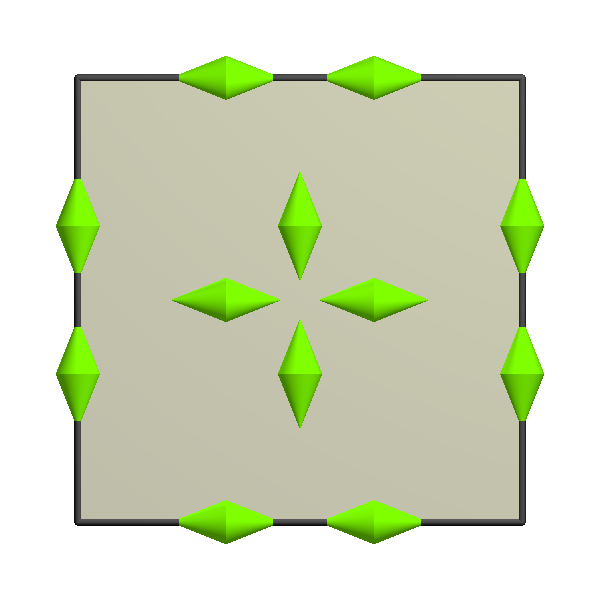}
  & \includegraphics[width=1.2in]{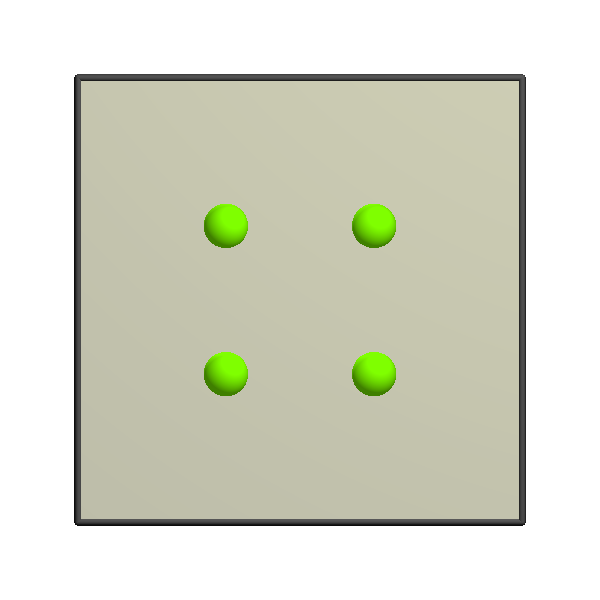}
\\
\raise.5in\hbox{$r=3$}
  & \includegraphics[width=1.2in]{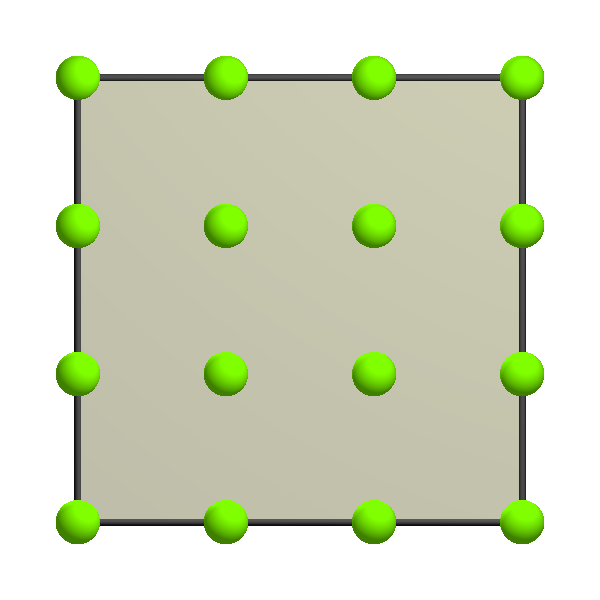}
  & \includegraphics[width=1.2in]{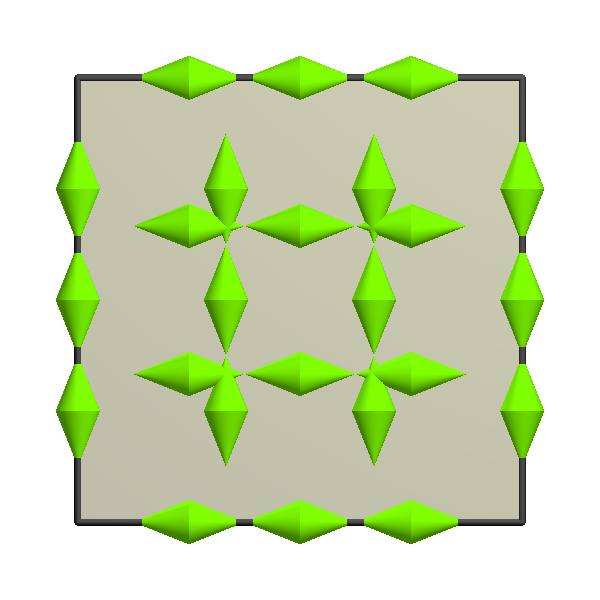}
  & \includegraphics[width=1.2in]{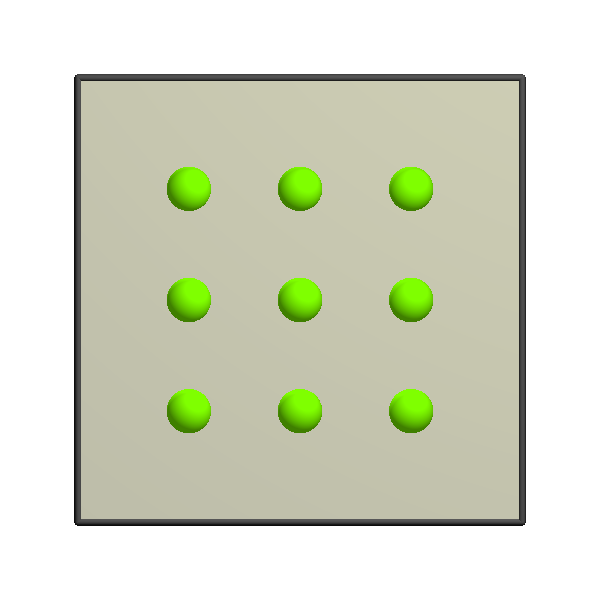}
 \end{tabular}
\begin{tabular}{ccccc}
$\Q_r^-\Lambda^k$ (3D): & $k=0$ & $k=1$ & $k=2$ & $k=3$
\\[.15in]
\raise.5in\hbox{$r=1$}
  & \includegraphics[width=1.2in]{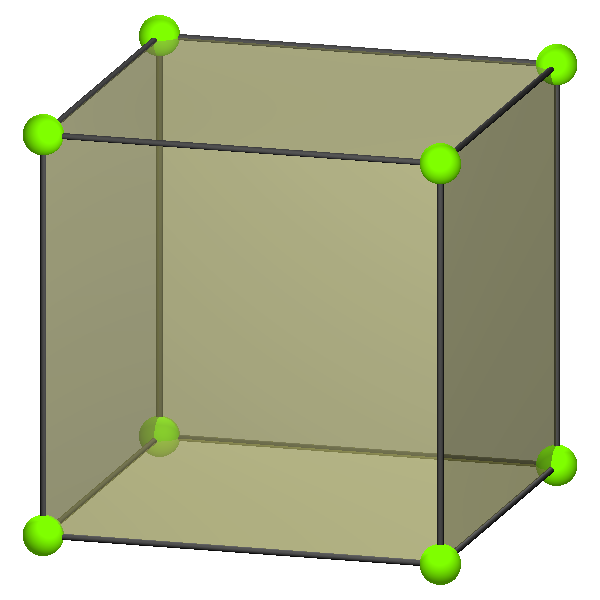}
  & \includegraphics[width=1.2in]{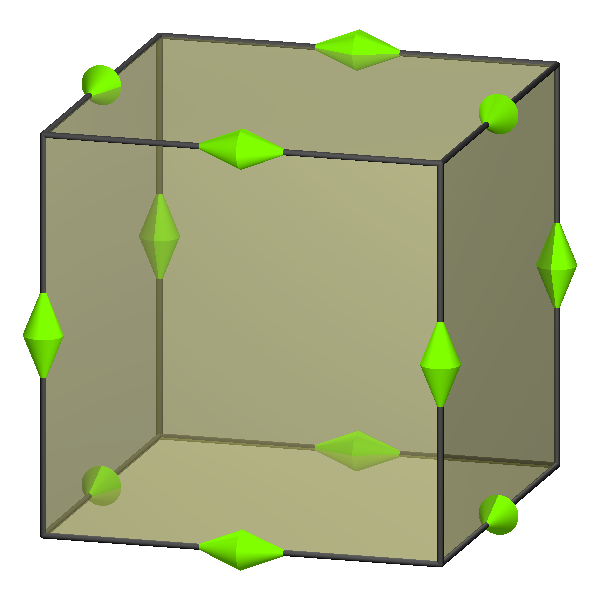}
  & \includegraphics[width=1.2in]{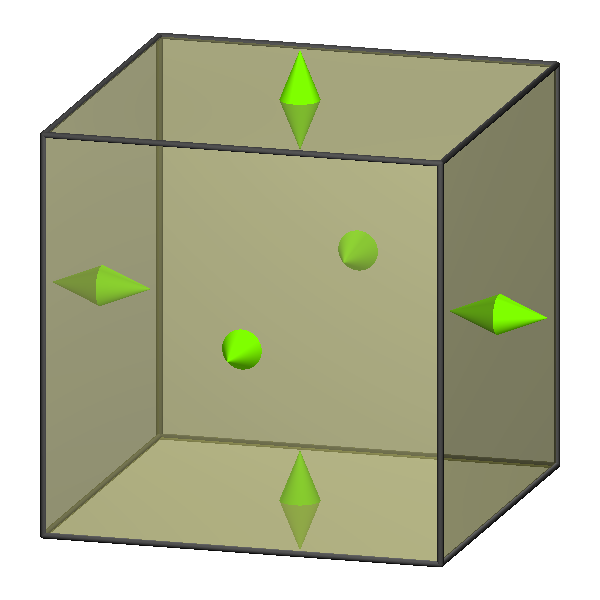}
  & \includegraphics[width=1.2in]{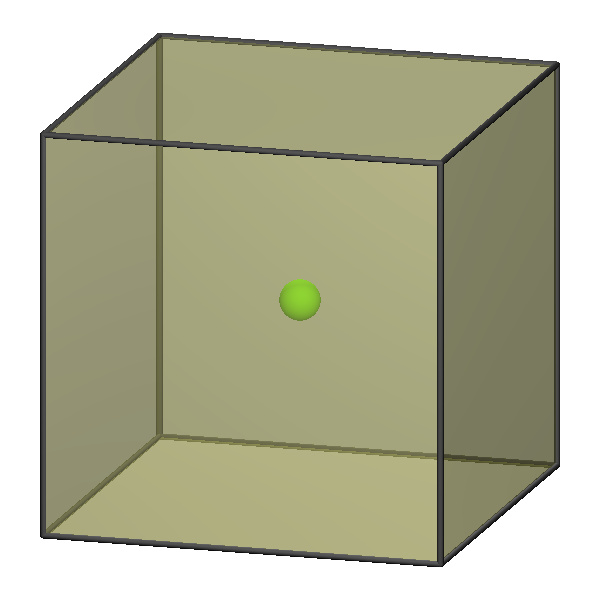}
\\
\raise.5in\hbox{$r=2$}
  & \includegraphics[width=1.2in]{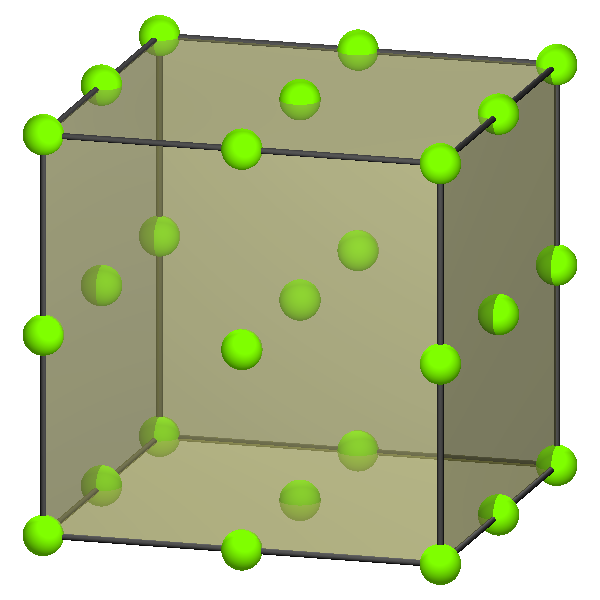}
  & \includegraphics[width=1.2in]{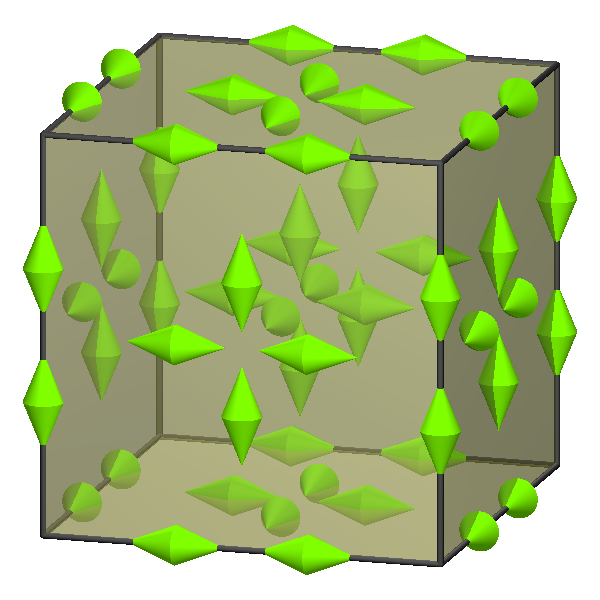}
  & \includegraphics[width=1.2in]{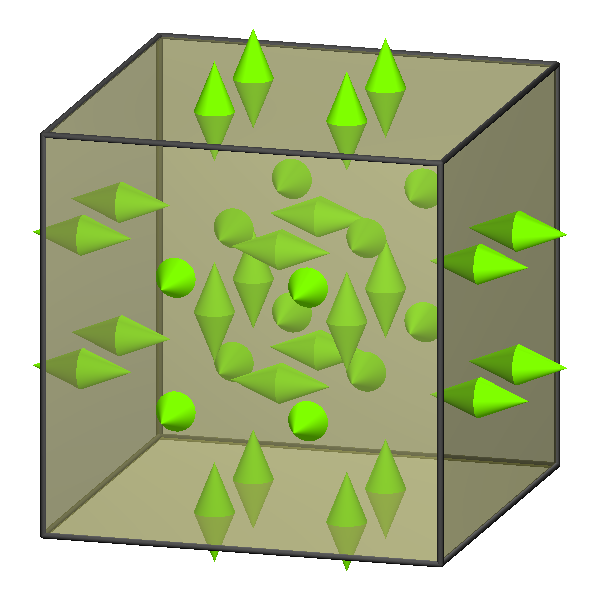}
  & \includegraphics[width=1.2in]{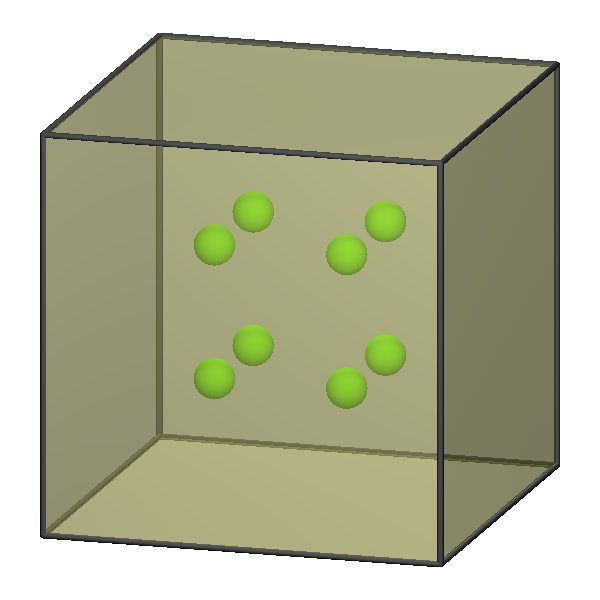}
\\
\raise.5in\hbox{$r=3$}
  & \includegraphics[width=1.2in]{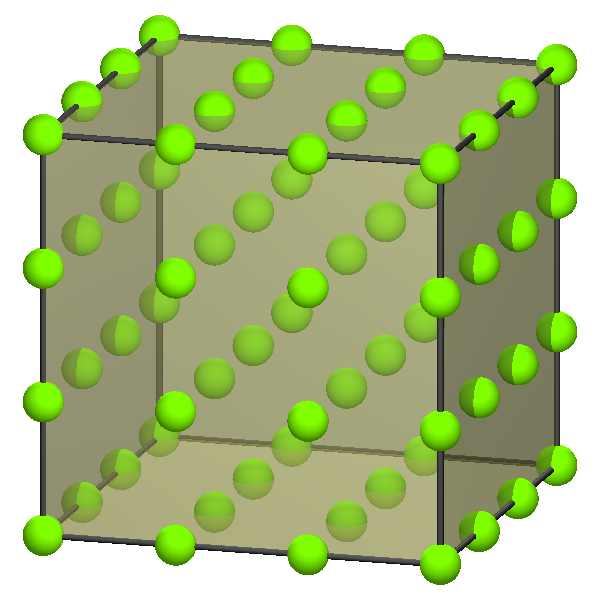}
  & \includegraphics[width=1.2in]{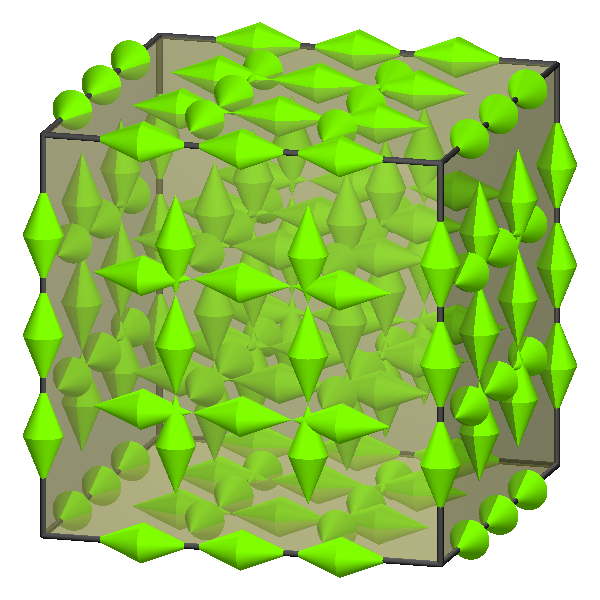}
  & \includegraphics[width=1.2in]{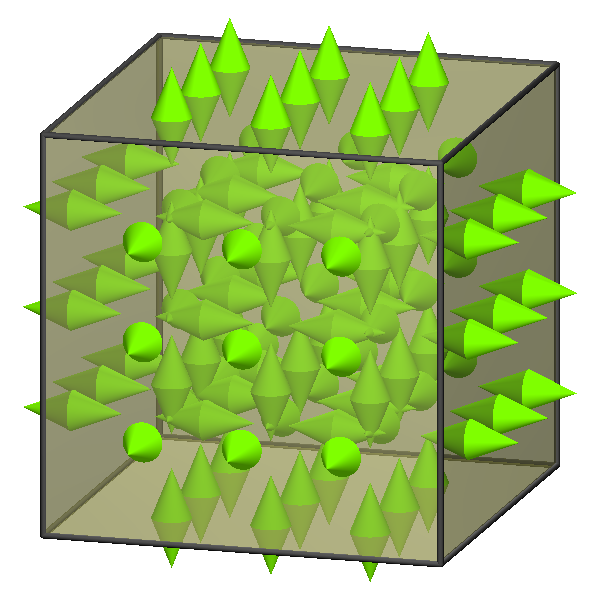}
  & \includegraphics[width=1.2in]{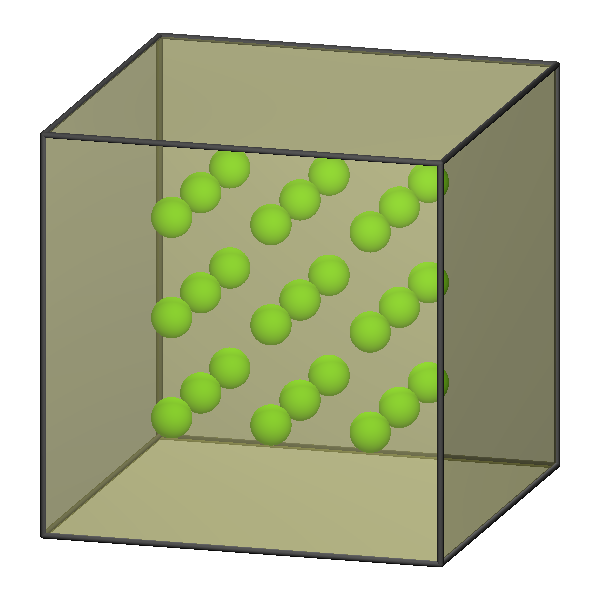}
 \end{tabular}
\end{center}
\caption{Degree of freedom diagrams for $\Q_r^-\Lambda^k$ spaces in two and three dimensions.}\label{fg:q}
\end{figure}
Now suppose that a mesh $\T$ of curvilinear cubes $K=F_K(\hat K)$ is given.  We assume that whenever two elements
$K_1$ and $K_2$ meet in a common face, say $f=F_{K_1}(\hat f_1)=F_{K_2}(\hat f_2)$, then the map
$F_{K_2}^{-1}\circ F_{K_1}|_{\hat f_1}$ mapping $\hat f_1$ onto $\hat f_2$ is linear.  The assembled
finite element space $\Q_r^-\Lambda^k(\T)$ consists of functions $u$ whose restriction $u_K$ to $K$ belongs to
$(F_K^{-1})^*\Q_r^-\Lambda^k(I^n)$ and for which the corresponding degrees of freedom are
single-valued on faces.  The choice of degrees of freedom implies that if $K_1$ and $K_2$ share a face $f$
of dimension $\ge k$, then the traces $\tr_fu_{K_1}$ and $\tr_f u_{K_2}$ coincide.  This is exactly the
condition needed for $u$ to belong to the space $H\Lambda^k(\Omega)$, i.e., for the exterior derivative
of $u$ to belong to $L^2\Lambda^{k+1}(\Omega)$ (see \cite[Lemma~5.1]{DRR}).  From the commutativity
of the pullback with the exterior derivative, we see that $d\Q_r^-\Lambda^k(\T)\subset \Q_r^-\Lambda^{k+1}(\T)$,
and so we obtain a subcomplex of the $L^2$ de~Rham complex on $\Omega$:
$$
\Q_r^-\Lambda^0(\T) \xrightarrow{d} \Q_r^-\Lambda^1(\T) \xrightarrow{d} \cdots 
\xrightarrow{d} \Q_r^-\Lambda^n(\T).
$$

The assembled finite element spaces $\Q_r^-\Lambda^k(\T)$ are well-known, especially when the 
maps $F_K$ are composed of dilations and translations, so the mesh consists of cubes.  The space $\Q_r^-\Lambda^0(\T)$ is
the usual $\Q_r$ approximation of $H^1(\Omega)$ and the space $\Q_r^-\Lambda^n(\T)$ is
the discontinuous $\Q_{r-1}$ approximation of $L^2(\Omega)$.  In three dimensions,
if we identify $H\Lambda^1$ and $H\Lambda^2$ with $H(\curl)$ and $H(\div)$, respectively,
then $\Q_r^-\Lambda^1(\T)$ and $\Q_r^-\Lambda^2(\T)$ are the N\'ed\'elec edge element
and face element spaces of the first kind, respectively \cite{Nedelec1}.

\section{Approximation properties on curvilinear cubes}\label{sc:approx}
We now return to the situation of Section~\ref{sc:fe} and
consider the approximation properties afforded by finite element spaces of differential $k$-forms
on meshes of curvilinear cubes.  Thus we suppose given:
\begin{enumerate}
 \item A space of reference shape $k$-forms $V(\hat K)$ and
a set of reference degrees of freedom $\Xi(\hat K)$ on the unit cube $\hat K$,
with each degree of freedom associated to a particular face of $\hat K$.
 \item A mesh $\T$ of the domain $\Omega$;
 \item For each $K\in \T$ a $C^1$ diffeomorphism $F_K$ of $\hat K$ onto $K$.
\end{enumerate}
We require regularity of the diffeomorphisms in the following sense.  Let $\bar K=h_K^{-1}K$
be the dilation of $K$ to unit diameter.  Then the composition $\bar F_K=h_K^{-1}F_K$ is
a $C^1$ diffeomorphism of $\hat K$ onto $\bar K$, and we assume $|\bar F_K|_{W^1_\infty(\hat K)}$
and $|\bar F_K^{-1}|_{W^1_\infty(\bar K)}$ are bounded uniformly for all elements.  We may therefore
apply Theorems~\ref{th:cov} and \ref{th:dil} to get
\begin{multline*}
 C^{-1} h_K^{k-n/2} \|v\|_{L^2\Lambda^k(K)}
  \le \|F^*_Kv\|_{L^2\Lambda^k(\hat K)} \le C h_K^{k-n/2} \|v\|_{L^2\Lambda^k(K)},
  \\
  v\in L^2\Lambda^k(K).
\end{multline*}

As described in Section~\ref{sc:fe}, the ingredients just given determine a space of $k$-form shape functions
$V(K)$ and a set of degrees of freedom $\Xi(K)$ on each element $K$.  From these, we
obtain an assembled finite element space $V(\T)$ consisting
of $k$-forms on $\Omega$ which belong piecewise to the $V(K)$ and for which corresponding
degrees of freedom associated to a common face of two elements take a single value.
We may define as well the set $\Xi(\T)$ of global degrees of freedom, each of
which is associated to a particular face $f$ of the triangulation.  If $\xi\in\Xi(\T)$ is
a global degree of freedom associated to a face $f$, and $K$ is any of the elements
containing $f$, then $\xi(v) = \hat\xi (F_K^* v|_K)$ where $\hat\xi\in\Xi(\hat K)$.
Since there exists a constant such that $|\xi(\hat v)|\le C\|\hat v\|_{L^2\Lambda^k(\hat K)}$
for all $\hat\xi\in\Xi(\hat K)$ and all $\hat v\in V(\hat K)$, we have
\begin{equation}\label{xibd}
|\xi(v)| \le C\|F_K^* v|_K\|_{L^2\Lambda^k(\hat K)}\le C h_K^{k-n/2}\|v\|_{L^2\Lambda^k(K)}.
\end{equation}

If it happens
that $V(K)$ contains the full polynomial space $\P_r\Lambda^k(K)$, then
Theorem~\ref{th:jackson} gives us the approximation result
\begin{equation}\label{approx}
\inf_{v\in V(K)} \|u-v\|_{L^2\Lambda^k(K)} \le C h_K^{r+1}|u|_{H^{r+1}\Lambda^k(K)},
\end{equation}
with a constant $C$ depending only on $r$, $n$, and the shape constant $C(K,r)$.
Our first goal in this section is to extend the estimate~\eqref{approx} to the global space $V(\T)$.
More precisely, in the setting just described and
under the same condition that for all $K\in\T$ the space
$V(K)$ contains the full polynomial space $\P_r\Lambda^k(K)$, we wish to show
that
\begin{equation}\label{glapprox}
\inf_{v\in V(\T)} \|u-v\|_{L^2\Lambda^k} \le C h^{r+1}|u|_{H^{r+1}\Lambda^k},
\end{equation}
where, as usual, $h$ denotes the maximum of the $h_K$.

In order to prove \eqref{glapprox}, we shall adapt the classical construction of the
Cl\'ement interpolant \cite{clement} to define an operator $\pi_h:L^2\Lambda^k\to V(\T)$ satisfying
\begin{equation}\label{clapprox}
\|u-\pi_h u\|_{L^2\Lambda^k} \le C h^{r+1}|u|_{H^{r+1}\Lambda^k},
\end{equation}
which certainly implies \eqref{glapprox}.   For
an analogous extension in the case of simplicial meshes, see~\cite[Sec.~5.4]{DRR}.

For the definition of the Cl\'ement interpolant we
need to introduce some notation.
Each degree of freedom $\xi\in\Xi(\T)$ determines a corresponding basis function $\phi_\xi\in V(\T)$ by
$\xi(\phi_\xi)=1$, $\eta(\phi_\xi)=0$ if $\eta\in\Xi(\T)$, $\eta\ne\xi$, giving the
expansion
\begin{equation}\label{basisexp}
v = \sum_{\xi\in\Xi(\T)} \xi(v)\phi_\xi, \quad v\in V(\T).
\end{equation}
Moreover, $\operatorname{supp}\phi_\xi\subset S_\xi$ where $S_\xi$ denotes
the union of elements $K$ which contain the face $f$.  The basis functions
on each element pull back to the reference basis functions, and so we can
obtain the estimate
\begin{equation}\label{basisbd}
 \|\phi_\xi\|_{L^2\Lambda^k(K)} \le C h_K^{-k+n/2},
\end{equation}
for all the basis functions and $K\in\T$.

If $S$ is a union of some of the elements of $\T$, let $P_S:L^2\Lambda^k(S)\to \P_r\Lambda^k(S)$
denote the $L^2$ projection.  If $S$ consists of the union of some of the elements which meet
a fixed element $K\in \T$, then we have
\begin{equation}\label{l2proj}
\|u-P_S u\|_{L^2\Lambda^k(S)} \le C h_K^{r+1}|u|_{H^{r+1}\Lambda^k(S)}.
\end{equation}
Indeed, applying the Bramble--Hilbert lemma on the
domain $S$, the left-hand side of (20) is bounded by
$C_S |u|_{H^{r+1}\Lambda^k(S)}$, and we need only show that the constant
$C_S$ can be bounded by $Ch_K$ with $C$ independent of $S$.
By translating and dilating we may assume that the
barycenter of $K$ is at the origin and its diameter $h_K=1$.
In that case, the number of vertices of the polygon $S$ is
bounded above in terms of the shape regularity constant, and for
each fixed number the point set determined by the vertices varies
in a compact set, again depending on the shape regularity.  Since the best
constant $C_S$ depends continuously on the the vertex set of $S$,
it is bounded on this compact set.
Motivated by \eqref{basisexp}, the Cl\'ement interpolant is defined by
$$
\pi_h u = \sum_{\xi\in \Xi(\T)} \xi(P_{S_\xi} u) \phi_\xi, \quad u\in L^2\Lambda^k.
$$

We now verify \eqref{clapprox}.
First we note that, restricted to a single element $K\in\T$,
$$
\pi_h u = \sum_{\xi\in \Xi(K)} \xi(P_{S_\xi} u)\phi_\xi.
$$
In view of \eqref{xibd} and \eqref{basisbd}, it follows that
\begin{equation}\label{l2bd}
\begin{aligned}
\|\pi_h u\|_{L^2\Lambda^k(K)}
 &\le \sum_{\xi\in\Xi(K)}|\xi(P_{S_\xi}u)| \|\phi_\xi\|_{L^2\Lambda^k(K)}
\\
&\le C\sum_{\xi\in\Xi(K)} h_K^{k-n/2}\|P_{S_\xi}u\|_{L^2\Lambda^k(K)}h_K^{-k+n/2}
\\
&\le C\sum_{\xi\in\Xi(K)}\|P_{S_\xi}u\|_{L^2\Lambda^k(K)}\le C\|u\|_{L^2\Lambda^k(\tilde K)}.
\end{aligned}
\end{equation}

Next, let $p=P_{\tilde K}u$, where $\tilde K$ is the union of all the elements of $\T$ meeting $K$.
Then $p$ is a polynomial form of degree at most $r$ on $\tilde K$, so $P_{S_\xi} p=p$  for all
$\xi\in\Xi(K)$.  It follows that, on $K$,
$$
\pi_h p = \sum_{\xi\in \Xi(K)} \xi(P_{S_\xi} p)\phi_\xi p  = \sum_{\xi\in \Xi(K)} \xi(p)\phi_\xi =p,
$$
with the last equality coming from \eqref{basisexp}.
Therefore, on $K$ we have that 
$$
\|u-\pi_h u\|_{L^2\Lambda^k(K)}   = \|u - p + \pi_h (u-p)\|_{L^2\Lambda^k(K)}
\le C\|u-p\|_{L^2\Lambda^k(\tilde K)},
$$
where we used \eqref{l2bd} in the last step (with $u$ replaced by $u-p$).  Applying \eqref{l2proj} in the case $S=\tilde K$
gives
$$
\|u-\pi_h u\|_{L^2\Lambda^k(K)}  \le C h_K^{r+1}|u|_{H^{r+1}\Lambda^k(\tilde K)},
$$
from which the global estimate \eqref{clapprox} follows.

In short, if the shape function space
$V(K) = (F_K^{-1})^*V(\hat K)$
contains  the full polynomial space $\P_r\Lambda^k(K)$, then the approximation estimate
\eqref{glapprox} holds.
Thus we are led to ask what conditions on the reference shape functions $V(\hat K)$ and the
mappings $F_K$ ensure that this is so.  Necessary and sufficient conditions on the space of
reference shape functions were given for multilinearly
mapped cubical elements in the
case of $0$-forms in two and three dimensions in \cite{DBF} and \cite{GM},
respectively.  For the case of $(n-1)$-forms,
i.e., $H(\div)$ finite elements, such conditions were
given in $n=2$ dimensions in \cite{DBF2}, in the lowest
order case $r=1$ in three dimensions in \cite{FalkGattoMonk}, and for
general $r$ in three dimensions in \cite{BD2}.
Necessary and sufficient conditions for $1$-forms in three dimensions
($H(\curl)$ elements) were also given in
the lowest order case in \cite{FalkGattoMonk}, and a closely
related problem for $H(\curl)$ elements 
in 3D was studied in \cite{BD1}.
These necessary and sufficient conditions in three dimensions are quite complicated,
and would be more so in the general case.  Here we content ourselves with sufficient
conditions which are much simpler to state and to verify, and which are adequate for many
applications.  These conditions are given in the following
theorem.  On the basis of the results given previously, its
proof is straightforward.
\begin{teo}\label{th:poly}
Suppose that either
\begin{enumerate}
 \item $F_K$ is an affine diffeomorphism and that $\P_r\Lambda^k(\hat K)\subset V(\hat K)$; or
 \item $F_K$ is a multilinear diffeomorphism and $\Q_{r+k}^-\Lambda^k(\hat K)\subset V(\hat K)$.
\end{enumerate}
Then $V(K)$ contains $\P_r\Lambda^k(K)$ and so \eqref{approx} and \eqref{glapprox} hold.
\end{teo}
\begin{proof}
 Let us write $F$ for $F_K$.
The requirement that $V(K)=(F^{-1})^*V(\hat K)$ contains $\P_r\Lambda^k(K)$ is
equivalent to requiring that $F^*\bigl(\P_r\Lambda^k(K)\bigr)\subset V(\hat K)$.
If $F$ is affine, then $F^*\bigl(\P_r\Lambda^k(K)\bigr)\subset \P_r\Lambda^k(\hat K)$,
as is clear from \eqref{pullback}.  The sufficiency of the first condition follows.

For $F$ multilinear, we wish to show that $F^*\bigl(\P_r\Lambda^k(K)\bigr)\subset\Q_{r+k}^-\Lambda^k(\hat K)$,
for which it suffices to show that $F^*(p\,dx^\sigma)\in \Q_{r+k}^-\Lambda^k(\hat K)$ if $p\in \P_r(K)$ and
$\sigma\in\Sigma(k,n)$. From \eqref{pullback} it suffices to show that
\begin{equation}\label{ts}
  (p\circ F) \pd{F^{i_1}}{\hat x^{j_1}}\cdots\pd{F^{i_k}}{\hat x^{j_k}}\,
  d\hat x^{j_1}\wedge\cdots\wedge d\hat x^{j_k} \in\Q_{r+k}^-\Lambda^k(\hat K).
\end{equation}
Now $p\in\P_r$ and $F$ multilinear imply that $p\circ F\in\Q_r(K)$.
Moreover, $\partial F^{i_l}/\partial \hat x^{j_l}$ is multilinear, but also independent of $\hat x^{j_l}$.
Therefore the product,\break
$(p\circ F) \partial{F^{i_1}}/{\partial\hat x^{j_1}}\cdots\partial{F^{i_k}}/{\partial\hat x^{j_k}}$, is a polynomial
of degree at most $r+k$ in all variables, but of degree at most $r+k-1$ in the variables $\hat x^{j_l}$.
Referring to the description of the spaces $\Q_r^-$ spaces derived in Section~\ref{sc:q}, we
verify \eqref{ts}.
\end{proof}

Note that the requirement on the reference space $V(\hat K)$ to obtain $O(h^{r+1})$ convergence
is much more stringent when the
maps $F_K$ are only assumed to be multilinear, than in the case when they are restricted
to be affine.  Instead of just requiring that $V(\hat K)$ contain the polynomial space $\P_r\Lambda^k(\hat K)$,
it must contain the much larger space $\Q_{r+k}^-\Lambda^k(\hat K)$.  For $0$-forms,
the requirement reduces to inclusion of the space $\Q_r(\hat K)$.  This result was obtained
previously in \cite{DBF} in two dimensions, and in \cite{GM}, in three dimensions.  The requirement
becomes even more stringent as the form degree, $k$, is increased.

\section{Application to specific finite element spaces}
Theorem~\ref{th:poly} gives conditions on the space $V(\hat K)$ of reference shape
functions which ensure a desired rate of $L^2$ approximation accuracy for the
assembled finite element space $V(\T)$.  In this section we consider several choices
for $V(\hat K)$ and determine the resulting implied rates of approximation.
According to the theorem, the result will be different for parallelotope meshes, in which
each element is an affine image of the cube $\hat K$, and for the more general situation of
curvilinear cubic meshes in which element is a multilinear image of $\hat K$ (which we
shall simply refer to as curvilinear meshes in the remainder the section).  Specifically,
if $s$ is the largest integer such that
\begin{equation}\label{p}
 \P_{s-1}\Lambda^k(\hat K)\subset V(\hat K),
\end{equation}
then the theorem implies a rate of $s$ (that is, $L^2$ error bounded by $O(h^{s})$) on parallelotope meshes, while if
$s$ is the largest integer for which the more stringent condition
\begin{equation}\label{c}
 \Q_{s+k-1}^-\Lambda^k(\hat K)\subset V(\hat K)
\end{equation}
holds, then the theorem implies a rate of $s$ more generally on curvilinear meshes.

\subsection{The space $\Q_r$}
First we consider the case of $0$-forms ($H^1$ finite elements) with
$V(\hat K)=\Q_r^-\Lambda^0(\hat K)$, which is the usual $\Q_r(\hat K)$ space, consisting of polynomials
of degree at most $r$ in each variable.  Then \eqref{p} holds for $s=r+1$, but not larger.  Therefore we find the $L^2$ approximation
rate to be $r+1$ on parallelotope meshes, as is well-known.  Since $k=0$, \eqref{c} also holds
for $s=r+1$, and so on curvilinear meshes we obtain
the same rate $r+1$ of approximation.  Thus, in this case, the generalization from parallelotope
to curvilinear meshes entails no loss of accuracy.

\subsection{The space $\Q_r^-\Lambda^n$}
The case of $n$-forms with $V(\hat K)=\Q_r^-\Lambda^n(\hat K)$ is quite different.
If we identify $n$-forms with scalar-valued functions, this is the space $\Q_{r-1}(\hat K)$.
Hence \eqref{p} holds with $s=r$, and so we achieve approximation
order $O(h^r)$ on parallelotope meshes.  However, for $k=n$, \eqref{c} holds
only if $s\le r-n+1$, and so Theorem~\ref{th:poly} only gives a rate of $r-n+1$ on curvilinear meshes in this
case, and \emph{no convergence at all} if $r\le n-1$.  Thus curvilinear meshes
entail a loss of one order of accuracy compared to parallelotope meshes in two dimensions,
a loss of two orders in three dimensions, etc.  This may seem surprising, since
if we identify an $n$-form $v\,dx^1\wedge\cdots\wedge dx^n$ with the function $v$, which
is a $0$-form, the space $\Q_r^-\Lambda^n(\hat K)$ corresponds to the space $Q_{r-1}(\hat K)=\Q_{r-1}^-\Lambda^0(\hat K)$,
which achieves the same rate $r$ on both classes of meshes.  The reason is that the
transformation of an $n$-form on $\hat K$ to an $n$-form on $K$ is very different
than the transform of a $0$-form.  The latter is simply $\hat v = v\circ F$ while the former
is
$$
\hat v \, d\hat x^1\wedge\cdots\wedge\hat x^n = (v\circ F) (\det F)\, dx^1\wedge \cdots \wedge dx^n.
$$

\subsection{The spaces $\Q_r^-\Lambda^k$}
Next we consider the case
$V(\hat K)=\Q_r^-\Lambda^k(\hat K)$ for $k$ strictly between $0$ and $n$.  Then
\eqref{p} holds with $s=r$ and \eqref{c} holds with $s=r-k+1$ (but no larger).
Thus, 
the rates of approximation achieved are $r$ on parallelotope meshes and
$r-k+1$ for curvilinear meshes.
In particular, if $r<k$, then the space provides no approximation in the curvilinear case.
For example, in three dimensions the space $\Q_1^-\Lambda^2(K)$, which, in conventional finite element
language is the trilinearly mapped Raviart-Thomas-N\'ed\'elec space, affords no approximation in $L^2$.
This fact was noted and discussed in \cite{nrw}.

\subsection{The space $\P_r\Lambda^n$}\label{ex-pr}
Since the space $H\Lambda^n(\Omega)$ does not require inter-element continuity,
the shape functions $V(\hat K)=\P_r\Lambda^n(\hat K)$ may be chosen
as reference shape functions (with all degrees of freedom in the interior
of the element).  In this case, \eqref{p} clearly holds for $s=r+1$,
giving the expected $O(h^{r+1})$ convergence on parallelotope meshes.
But \eqref{c} holds if and only if $\Q_{s+n-2}(\hat K)\subset\P_r(\hat K)$,
which holds if and only if $n(s+n-2)\le r$.  In two dimensions this condition
becomes $s\le r/2$, and so we only obtain the rate of $\lfloor r/2\rfloor$,
and no approximation at all for $r=0$ or $1$.  In three dimensions,
the corresponding rate is $\lfloor r/3\rfloor-1$, requiring $r\ge 6$ for first
order convergence on curvilinear meshes.  For general $n$, the rate is $\lfloor r/n\rfloor -n+2$.

\subsection{The serendipity space $\S_r$}\label{ex-serendipity}
The serendipity space $\S_r$, $r\ge 1$, is a finite element subspace of $H^1$ (i.e., a space of $0$-forms).
In two dimensions and, for small $r$, in three dimensions, the space has been used for many
decades.  In 2011, a uniform definition was given for all dimensions $n$ and all degrees $r\ge 1$ \cite{serendipity}.
According to this definition, the shape function space $\S_r(\hat K)$ consists of all polynomials of \emph{superlinear
degree} less than or equal to $r$, i.e., for which each monomial has degree at most $r$ ignoring those
variables which enter the monomial linearly (example: the monomial $x^2 y z^3$ has superlinear degree $5$).
From this definition, it is easy to see that $\P_r(\hat K)\subset \S_r(\hat K)$ but $\P_{r+1}(\hat K)\not\subset\S_r(\hat K)$.
Thus, from \eqref{p}, the rate of $L^2$ approximation of $\S_r$  on parallelotope meshes
is $r+1$.  Now $\Q_1(\hat K)\subset\S_r(\hat K)$
for all $r\ge 1$, but for $s\ge 2$, $\Q_s(\hat K)$ contains the element $(x_1\cdots x_n)^s$ of superlinear
degree $n s$.  It follows that \eqref{c} holds if and only if $s\le \max(2,\lfloor r/n\rfloor +1)$, which gives the
$L^2$ rate of convergence of the serendipity elements on curvilinear meshes.
This was shown in two dimensions in \cite{DBF}.

\subsection{The space $\S_r\Lambda^k$}
In \cite{cubicderham}, Arnold and Awanou defined shape functions and degrees of freedom for
a finite element space $\S_r\Lambda^k$ on cubic meshes in $n$-dimensions
for all form degrees $k$ between $0$ and $n$,
and polynomial degrees $r\ge 1$.  The shape function space $\S_r\Lambda^k(\hat K)$
they defined contains $\P_r\Lambda^k(\hat K)$ but not $\P_{r+1}\Lambda^k(\hat K)$, so the assembled finite element space
affords a rate of approximation $r+1$ on parallelotope meshes.
In the case $k=n$, $\S_r\Lambda^n(\hat K)$ in fact coincides with $\P_r\Lambda^n(\hat K)$,
so, as discussed above, the rate is reduced to $\lfloor r/n\rfloor -n+2$ on curvilinear
meshes.   In the case $k=0$, $\S_r\Lambda^0$ coincides with the serendipity
space $\S_r$, and so the rate is reduced to $\max(2,\lfloor r/n\rfloor +1)$ on curvilinear
meshes in that case.
In $n=2$ dimensions, this leaves the space $\S_r\Lambda^1$,
which is the BDM space on squares, for which the reference shape functions
are $\P_r\Lambda^1(\hat K)$ together with the span of the two forms $d(x^{r+1}y)$ and $d(xy^{r+1})$.
To check condition \eqref{c} we note that if $2s-1\le r$, then
$\Q_s^-\Lambda^1(\hat K)\subset\P_r\Lambda^1(\hat K)\subset \S_r\Lambda^1(\hat K)$.  However, if $2s-1 > r$,
then the $1$-form $x^{s-1}y^s\,dx$ belongs to $\Q_s^-\Lambda^1(\hat K)$ but not to
$\S_r\Lambda^1(\hat K)$.  We conclude that the rate of approximation
of $\S_r\Lambda^1$
on quadrilateral meshes in two dimension is $\lfloor (r+1)/2\rfloor$.

\bibliographystyle{amsplain}
\bibliography{quadfeec}

\end{document}